\documentclass[12pt,english]{article}
\usepackage[T1]{fontenc}
\usepackage[latin9]{inputenc}
\usepackage{amsthm}
\usepackage{amsmath}

\makeatletter
\theoremstyle{plain}
 \theoremstyle{definition}
 \newtheorem*{defn*}{Definition}
\theoremstyle{plain}
\newtheorem{thm}{Theorem}
  \theoremstyle{remark}
  \newtheorem*{rem*}{Remark}
  \theoremstyle{remark}
  \newtheorem*{acknowledgement*}{Acknowledgement}
  \theoremstyle{plain}
  \newtheorem{lem}[thm]{Lemma}
  \theoremstyle{plain}
  \newtheorem{cor}[thm]{Corollary}
  \theoremstyle{definition}
  \newtheorem{defn}[thm]{Definition}

\usepackage{amssymb}

\usepackage{latexsym}
\usepackage{graphics}
\newcommand{\ben}{\begin{enumerate}}
\newcommand{\een}{\end{enumerate}}
\newcommand{\ble}{\begin{lem}}
\newcommand{\ele}{\end{lem}}
\newcommand{\bth}{\begin{thm}}
\renewcommand{\eth}{\end{thm}}
\newcommand{\bpr}{\begin{prop}}
\newcommand{\epr}{\end{prop}}
\newcommand{\bco}{\begin{cor}}
\newcommand{\eco}{\end{cor}}
\newcommand{\bcon}{\begin{conj}}
\newcommand{\econ}{\end{conj}}
\newcommand{\bde}{\begin{defn}}
\newcommand{\ede}{\end{defn}}
\newcommand{\bex}{\begin{exa}}
\newcommand{\eex}{\end{exa}}
\newcommand{\barr}{\begin{array}}
\newcommand{\earr}{\end{array}}
\newcommand{\btab}{\begin{tabular}}
\newcommand{\etab}{\end{tabular}}
\newcommand{\beq}{\begin{equation}}
\newcommand{\eeq}{\end{equation}}
\newcommand{\bea}{\begin{eqnarray*}}
\newcommand{\eea}{\end{eqnarray*}}
\newcommand{\bce}{\begin{center}}
\newcommand{\ece}{\end{center}}
\newcommand{\bpi}{\begin{picture}}
\newcommand{\epi}{\end{picture}}
\newcommand{\bfi}{\begin{figure} \begin{center}}
\newcommand{\efi}{\end{center} \end{figure}}

\newcommand{\bsl}{\begin{slide}{}}
\newcommand{\esl}{\end{slide}}

\newcommand{\hso}[1]{\hspace{-1pt}}

\newcommand{\sbe}{\subseteq}

\newcommand{\iso}{\cong}

\def\<{\langle}
\def\>{\rangle}

\newcommand{\cA}{{\cal A}}

\newcommand{\cO}{{\cal O}}

\newcommand{\gen}[1]{\langle #1 \rangle}

\newcommand{\inc}{\mathop{\rm inc}\nolimits}

\newcommand{\rad}{\mathop{\rm rad}\nolimits}

\newcommand{\rank}{\mathop{\rm rank}\nolimits}

\renewcommand{\inc}{\star}
\newcommand{\typ}{\mathop{\rm typ}}%

\newcommand{\res}{{\rm res \thinspace }}

\newcommand{\Res}{\mathop{\rm Res}\nolimits}

\newcounter{romanlistctr}

 \makeatletter
 \@addtoreset{equation}{section}
 \makeatother

 \makeatletter
 \def\section{\@startsection {section}{1}{\z@}{-1.5ex plus -.5ex
 minus -.2ex}{1ex plus .2ex}{\large\bf}}

 \makeatletter
 \def\subsection{\@startsection {subsection}{1}{\z@}{-1.5ex plus -.5ex
 minus -.2ex}{1ex plus .2ex}{\bf}}

\newcommand{\Fq}{\mathbb{F}_q}

\newcommand{\set}[2]{ 
        {\left\{\left.
        #1\vphantom{#2\bigl(\bigr)}\,\right|
        \,#2\right\}}}

\newcommand{\four}[4]{
\left(
\begin{array}{cc}
#1 & #2 \\
#3 & #4
\end{array}
\right) }

\makeatother

\usepackage{babel}

\begin{document}
\title{On a quasi-Phan theorem for orthogonal groups}
\author{ {\sc Corneliu Hoffman} \\ {School of Mathematics, University of Birmingham}\\ {Edgbaston Birmingham B15 2TT, United Kingdom}\\ {\tt C.G.Hoffman@bham.ac.uk} \and {\sc Adam Roberts} \\ {Department of Mathematics} \\ {Clarion University of Pennsylvania}\\ {840 Wood Street, Clarion, PA 16214-1232, USA } \\ {\tt aroberts@clarion.edu }}\maketitle

\section{Introduction}

The Curtis-Tits-Phan theory was introduced in \cite{BenShp04,BenGraHofShp03}
to obtain amalgam presentations for some of the classical groups by
geometric means. Such results are useful in the revision project of
the Classification of Finite Simple Groups by Lyons and Solomon. A
series of papers \cite{GraHofShp03,BenShp04,Gra04,GraHofNicShp05},
provided such recognition theorems. The idea is to find a simply connected
geometry on which the group $G$ acts and then use a lemma of Tits
to conclude that the group $G$ is the universal completion of the
amalgam of maximal parabolics.

All the examples cited above use the concept of a flip. These are
involuntary automorphisms of certain twin buildings. We refer to \cite{BenGraHofShp03} for
the details of the general construction. For the purposes of this
paper we will consider a flip induced by an orthogonal polarity. The
geometry is similar to the one in \cite{BenShp04} but the main trouble
here is that the orthogonal group does not act flag transitively on
the \textquotedbl{}flip-flop\textquotedbl{} geometry and so one needs
to consider orbits of this action. A different approach to this problem
was considered in \cite{GraVan06} using a  version of Tits' Lemma
for intransitive actions. A similar modification of the {}``flip-flop''
geometry for symplectic groups can be found in \cite{BloHof08}.

To fix the notations we will consider $\Fq$ a finite field with $q$
elements and we will assume that $q$ is odd and $-1$ is a square
in $\Fq$. Let $V$ be an $n+1$-dimensional vector space over $\Fq$
with a nondegenerate symmetric bilinear form $(\cdot,\cdot)$ that
admits an orthonormal basis (i.e. a form of $+$-type). Similar results
can be deduced for the case when $-1$ is not a square and respectively
if the form is minus-type but the separate analysis is too long. Most
of it can be can be found in \cite{Rob05} and it would appear in
subsequent articles. The case $q$ even is however quite different
and the geometry must be significantly modiffied. It has been considered
by N. Iverson is his forthcoming thesis \cite{Ive09}.

Define $\Gamma'=\set{0<W<V}{W\mbox{ is nondegenerate}}$. This collection
is a pre-geometry with type set $\{1,\ldots,n\}$. Points are the
nondegenerate $1$-dimensional subspaces, lines are the nondegenerate
$2$-dimensional subspaces, etc. Recall that given a basis $\{v_{i}\}$
for a subspace of $V$ we can compute the Gram matrix $G=\left(g_{ij}\right)=(v_{i},v_{j})$.
The Orthogonal group does not act flag-transitively on this geometry.
We will describe one of the orbits.
\begin{defn*}
A square-type subspace $W$ is a nondegenerate subspace of $V$ such
that any basis for $W$ gives a Gram matrix with square determinant.
\\
 A nonsquare-type subspace $W'$ is a nondegenerate subspace of
$V$ such that any basis for $W'$ gives a Gram matrix with nonsquare
determinant. 
\end{defn*}
If the Gram matrix for one basis is a square then the Gram matrix
for any basis is also a square and similarly if the Gram matrix is
a nonsquare. So a square-type subspace and a nonsquare-type subspace
are well-defined concepts. Finally we will sometimes find it useful
to use the more traditional terminology for orthogonal spaces and
use the term plus-type line to refer to $2$-dimensional (nondegenerate)
subspace with a hyperbolic basis and minus-type line to refer to a
$2$-dimensional (nondegenerate) definite subspace. In the case considered,
square-type and plus-type are the same thing.

Our geometry is then $\Gamma=\set{0<W<V}{W\mbox{ is nondegenerate of square-type}}$.

We now state the results. 
\begin{thm}
\label{thm1} If $n\ge3$ or $n=2$ and $q\ge7$ the geometry is connected.
If $n\ge4$ and $q\ge7$ or $n=3$ and $q\ne5,9,13,17,25,29,37,41,53,73$
the geometry is simply connected. If $n\ge4$ and $ $$q\ne5,9,13,17,25,29,37,41,53,73$
the geometry is residually connected and residually simply connected. 
\end{thm}
Applying Tits' lemma \ref{thm:[Tits'-Lemma]} we get
\begin{thm}
\label{thm2} Let $G=SO(V)$ acting on the geometry $\Gamma$ and
${\cal A}_{k}$ the corresponding amalgam of parabolics of rank at
most $k$. Then\end{thm}
\begin{enumerate}
\item If $n\ge4$ and $q\ge9$ or $n=3$ and $q\ne5,9,13,17,25,29,37,41,53,73$.
then $G$ is the universal completion of $\cA$. 
\item If $n\ge4$ and $q\ne5,9,13,17,25,29,37,41,53,73$ then $G$ is the
universal completion of $\cA_{2}$.\end{enumerate}
\begin{rem*}
The result improves on \cite{GraVan06}. More precisely, our amalgam
$\cA$ (respectively $\cA_{3}$) contains less subgroups than the
one in \cite{GraVan06} giving a more efficient presentation. Moreover
the conditions on $q$ are identical for $n\ge4$ and only a little
worse for the case $n=3$. Finally we do not know whether the exceptional
cases are genuine exceptions. In particular the cases $q=25,41,53,73$
are likeley to give a simply connected geometry.
\end{rem*}

\begin{rem*}
The second part of Theorem \ref{thm2} states that the orthogonal
group can be recognized from an amalgam of orthogonal groups in dimension
two and three only. One would be tempted to look for a slightly stronger
result, in terms of {}``Phan''-systems as in \cite{BenShp04}. That
is to say that the group is the universal completion of any amalgam
of rank one and rank two parabolics with only the local information.
Unfortunately the amalgam is not uniquely determined by the rank $1$
and $2$ subgroups alone. Indeed the amalgam $\cA_{2}$ is almost
like a Coxeter presentation (since the rank one parabolics are dihedral).
There are many nonisomorphic such amalgams. For example if $n=3$,
consider the geometry whose objects of type one and two are square
type subspaces of the corresponding dimension but objects of type
three are nonsquare-type subspaces of dimension three. The amalgam
of maximal parabolics for this geometry has the same {}``Phan''-structure
as our amalgam but it is not isomorphic to it. Theorem \ref{thm2} states
that the orthogonal group is the universal completion of the one amalgam
that lies inside it. It seems that if one is willing to enlarge the
amalgam with some extra pieces of the amalgam in \cite{GraVan06}
this problems goes away and the amalgam becomes unique.
\end{rem*}
The article is structured as follows. In Section \ref{geometry} we
introduce the relevant geometry and amalgam notions, in Section \ref{prelim}
we collect a few technical lemmas to be used in the rest of the paper.
In Section \ref{con} we prove the connectedness results in Theorem
\ref{thm1}. Section \ref{nge4} and \ref{n=00003D3} finish the proof
of Theorem \ref{thm1} in the case $n\ge4$ and $n=3$ respectively. 
\begin{acknowledgement*}
The higher rank results in this article are part of the second author's
Ph.D. thesis \cite{Rob05} under the supervision of the first author.
Also we would like to thank Joe Wetherell for supplying us with a
proof of Lemma \ref{lem:Joe's lemma} allowing us to deal with the
small rank case hence improving the results.
\end{acknowledgement*}

\section{Geometries, Amalgams and Tits' Lemma}

\label{geometry}

\subsection{Geometries}

For our viewpoint on geometries we'll use the following definitions
from Buekenhout~\cite{Bue95}.

A \emph{pre-geometry} over a \emph{type set} $I$ is a triple $\Gamma=(\cO,\typ,\inc)$,
where $\cO$ is a collection of \emph{objects} or \emph{elements},
$I$ is a set of \emph{types}, $\inc$ is a binary symmetric and reflexive
relation, called the \emph{incidence relation} and $\typ\colon\cO\to I$
is a \emph{type function} such that whenever $X\inc Y$, then either
$X=Y$ or $\typ(X)\ne\typ(Y)$.

The \emph{rank} of the pre-geometry $\Gamma$ is the size of $\typ(\cO)$.
A \emph{flag} $F$ is a (possibly empty) collection of pairwise incident
objects. Its \emph{type} (resp. \emph{cotype}) is $\typ(F)$ (resp.
$I-\typ(F)$). The \emph{rank} of $F$ is $\rank(F)=|\typ(F)|$.
The \emph{type} of $F$ is $\typ(F)=\{\typ(X)\mid X\in F\}$. A \emph{chamber}
is a flag $C$ of type $I$.

A pre-geometry $\Gamma$ is a \emph{geometry} if $\typ(\cO)=I$ and
if $\Gamma$ is \emph{ transversal}, that is, if any flag is contained
in a chamber.

The \emph{incidence graph} of the pre-geometry $\Gamma=(\cO,\typ,\inc)$
over $I$ is the graph $(\cO,\inc)$. This is a multipartite graph
whose parts are indexed by $I$. We call $\Gamma$ \emph{connected}
if its incidence graph is connected.

The \emph{ residue} of a flag $F$ is the pre-geometry $\Res_{\Gamma}(F)=(\cO_{F},\typ|_{\cO_{F}},\inc|_{\cO_{F}})$
over $I-\typ(F)$ induced on the collection $\cO_{F}$ of all objects
in $\cO-F$ incident to $F$. We call $\Gamma$ \emph{residually connected}
if for every flag of rank at least $2$ the corresponding residue
is connected.

For a subset $K\sbe I$ the $K$-\emph{shadow} of a flag $F$ is the
collection of all $K$-flags incident to $F$.

Let $\Gamma$ be a connected geometry over the finite set $I$. A
\emph{ path of length $k$} is a path $x_{0},\ldots,x_{k}$ in the
incidence graph. We do not allow repetitions, that is, $x_{i}\ne x_{i+1}$
for all $0\le i<k$. A \emph{cycle based at an element $x$} is a
path $x_{0},\ldots,x_{k}$ in which $x_{0}=x=x_{k}$. Two paths $\gamma$
and $\delta$ are \emph{homotopically equivalent} if one can be obtained
from the other by inserting or eliminating cycles of length $2$ (returns)
or $3$ (triangles). We denote this by $\gamma\simeq\delta$. The homotopy
classes of cycles based at an element $x$ form a group under concatenation.
This group is called the \emph{fundamental group of $\Gamma$ based
at $x$} and is denoted $\Pi_{1}(\Gamma,x)$. If $\Gamma$ is (path)
connected, then the isomorphism type of this group does not depend
on $x$ and we call this group simply the \emph{fundamental group}
of $\Gamma$ and denote it $\Pi_{1}(\Gamma)$. We call $\Gamma$ \emph{simply
connected} if $\Pi_{1}(\Gamma)$ is trivial. We call $\Gamma$ \emph{residually
simply connected} if all its residues are simply connected.

\subsection{Amalgams}

While the concept of a geometry is quite interesting in its own right,
the main reason for introducing it is to study the associated group
amalgam. We start with a few definitions and the statement of Tits'
Lemma. For a more detailed treatment see \cite{Ser03}
\begin{enumerate}
\item An amalgam of groups is a set ${\cal A}$ together with a collection
of subsets $\{G_{i}\}_{i\in I}$ so that ${\cal A}=\cup_{i\in I}G_{i}$
and a partially defined multiplication operation $*$ such that 

\begin{enumerate}
\item $\forall i,*:G_{i}\times G_{i}\mapsto G_{i}$ makes $G_{i}$ defines
a group structure for $G_{i}$. 
\item $\forall i,j, G_{i}\cap G_{j}$ is a subgroup of both $G_{i}$ and $G_{j}$. 
\item If $a,b\in{\cal A}$ then one can define $a*b$ if and only if $\exists i\in I$
with $a,b\in G_{i}$. 
\end{enumerate}
\item A completion of an amalgam ${\cal A}=\cup_{i}G_{i}$is group $G$
together with group homeomorphisms $\phi_{i}:G_{i}\mapsto G$ so that
$G$ is generated by the images of the $\phi_{i}$
\item A universal completion of an amalgam ${\cal A}$ is a completion $\tilde{G}$
so that for any other completion $G$ there exists a unique surjective
homomorphism $\Phi:\tilde{G}\mapsto G$ compatible with each of the
maps from the amalgam. 
\end{enumerate}
Given a geometry $\Gamma$ and a group $G$ acting flag transitively
on $\Gamma$ one can define an amalgam as follows. Fix $c=\{x_{i}\}_{i\in I}$
a maximal flag of $\Gamma$ and define $G_{i}=Stab_{G}(x_{i})$ the
stabiliser of the object of type $i$. Then ${\cal A}=\cup_{I\in I}G_{i}$
is called the amalgam of maximal parabolics associated to the geometry
$\Gamma$. The following theorem connects the two notions (see \cite{Tit86})
: 
\begin{thm}
{[}Tits' Lemma{]}\label{thm:[Tits'-Lemma]} Let $\Gamma$ be a connected
geometry and $G$ be a flag transitive group of automorphisms of $\Gamma$.
The geometry is simply connected if and only if $G$ is isomorphic
to the universal completion of the amalgam of maximal parabolics associated
to $\Gamma.$ 
\end{thm}
As above assume $\Gamma$ is a geometry and $G$ a group of automorphisms
that acts flag transitively. Consider a chamber $c$ and the amalgam
$\cA_{k}$ of all the stabilisers of subflags of rank $k$ of $c$.
Inductively one can prove the following theorem.
\begin{thm}
$G$ is the universal completion of $\cA_{k}$ if and only if all
the residues of rank more than $k$ are simply connected. 
\end{thm}

\section{Preliminaries}

\label{prelim} Here we collect some technical lemmas. We first start
with a lemma about finite fields. We would like to thank Joe Wetherell
for supplying us with the proof
\begin{lem}
\label{lem:Joe's lemma} Assume $q$ is odd and $q>47$. Then for
any $c\in\Fq$ so that $c^{2}+1$ is a nonzero square there exist
elements $a,b\in\Fq^{*}$ so that $a^{2}+1,b^{2}+1$ are all nonzero
squares in $\Fq$ and $c^{2}=a^{2}+b^{2}$. \end{lem}
\begin{proof}
Consider the equations \begin{eqnarray}
A^{2} & = & a^{2}+1\\
B^{2} & = & b^{2}+1\\
C^{2} & = & c^{2}+1\\
c^{2} & = & a^{2}+b^{2}\end{eqnarray}
 where we want $a,b,c,A,B,C$ to all be finite and non-zero elements
of $\Fq.$

Forget equation (4) and the finite, non-zero condition for the moment.
The solutions to equations (1), (2), and (3) are parametrized by arbitrary
triples $(x,y,z)$ via

\[
a=\frac{2x}{1-x^{2}}\ A=\frac{1+x^{2}}{1-x^{2}}\]
 \[
b=\frac{2y}{1-y^{2}}\ B=\frac{1+y^{2}}{1-y^{2}}\]
 \[
c=\frac{2z}{1-z^{2}}\ C=\frac{1+z^{2}}{1-z^{2}}\]

The finite and non-zero condition is exactly satisfied if each of
$x,y,z$ avoid the 6 values $\infty,0,1,-1,i,-i$ So this defines an
open subset $U$ of the affine 3-space.

Within $U$ we want the surface satisfying the equation $c^{2}=a^{2}+b^{2}$;
in other words:

\[
\frac{4z^{2}}{(1-z^{2})^{2}}=\frac{4y^{2}}{(1-y^{2})^{2}}+\frac{4x^{2}}{(1-x^{2})^{2}}\]
 Suppose that $q>413$, and choose the value $z_{0}$ corresponding
to $c$. Define $e=\frac{z_{0}^{2}}{(1-z_{0}^{2})^{2}}$. Setting $z=z_{0}$
gives us the equation for a curve; clearing denominators and dividing
by 4, this equation is

\[
ex^{4}y^{4}-(2e+1)x^{4}y^{2}+ex^{4}-(2e+1)x^{2}y^{4}+(4e+4)x^{2}y^{2}-(2e+1)x^{2}+ey^{4}-(2e+1)y^{2}+e=0\]

This curve has arithmetic genus $9$, so, by the Hasse-Weil the bound of the
number of rational points over $\Fq$ is at least $q+1-18\sqrt{q}$.
We need to avoid at most $6$ $x$-values and $6$ $y$-values. Each $x$-value occurs
in at most $4$ points and likewise for $y$-values, so we need to avoid
at most $48$ points. Since $q>413$, a quick computation shows that
$q+1-18\sqrt{q}>48.1>48.$

It follows that there is a least one allowed point on the $z=z_{0}$
curve, so the surface also has at least one point.

Suppose that $q<413$

One can check remaining $91$ values of $q$ by computer -- the $91$ values
are

3, 5, 7, 9, 11, 13, 17, 19, 23, 25, 27, 29, 31, 37, 41, 43, 47, 49,
53, 59, 61, 67, 71, 73, 79, 81, 83, 89, 97, 101, 103, 107, 109, 113,
121, 125, 127, 131, 137, 139, 149, 151, 157, 163, 167, 169, 173, 179,
181, 191, 193, 197, 199, 211, 223, 227, 229, 233, 239, 241, 243, 251,
257, 263, 269, 271, 277, 281, 283, 289, 293, 307, 311, 313, 317, 331,
337, 343, 347, 349, 353, 359, 361, 367, 373, 379, 383, 389, 397, 401,
409

We find that if $q\not\in[3,5,7,9,11,13,17,19,23,25,27,29,31,37,41,43,47,53,59,61,73,103]$
the problem has solution. Therefore the only bad cases are $q=5,9,13,17,25,29,37,41,53,73$.
Even in these cases if $q\ge37$, about half the coices for $z$ will
work. 

We now prove a few elementary but useful results about small dimensional
subspaces.\end{proof}
\begin{lem}
On any plus-type line there are $\frac{q-1}{2}$ square-type points,
$\frac{q-1}{2}$ nonsquare-type points and $2$ degenerate points.
On any minus-type line there are $\frac{q+1}{2}$ square points and
$\frac{q+1}{2}$ nonsquare points. \end{lem}
\begin{proof}
This is an elementary computation inside the dihedral groups $O_{2}^{\epsilon}(q)$. \end{proof}
\begin{cor}
\label{cor:sum of squares} If $\alpha\in\Fq$ then there are $q-1$
pairs of $x$ and $y$ such that $x^{2}+y^{2}=\alpha$. \end{cor}
\begin{proof}
Let W be a two dimensional orthogonal space of square-type. Let $\{a,b\}$
be an orthonormal basis for $W$. Since $-1$ is a square then there
are $(q-1)/2$ square-type points in $W$. In particular there are
$(q-1)/2$ linearly independent choices of $xa+yb$ such that $x^{2}+y^{2}=1$.
Notice the only scalars we can multiply $xa+yb$ by without changing
the norm are $\pm1$ and so there are $q-1$ choices of $x$ and $y$
which work.\end{proof}
\begin{lem}
\label{lem:no tricky subspaces} Let $W$ be a nondegenerate $3$-dimensional
subspace of $V$. If $W$ is of square-type then any degenerate $2$-subspace
of $W$ contains $q$ square-type points. If $W$ is of nonsquare-type,
then any degenerate $2$-subspace of $W$ contains $q$ nonsquare-type
points. \end{lem}
\begin{proof}
Assume that $W$ is of square-type and let $U$ a degenerate subspace
of dimension two. Assume $a\in U$ is nonsquare-type and $r$ is the
radical of $U$. It follows that $r\in a^{\perp}\cap W$. However
since $W$ is square-type, $a^{\perp}\cap W$ is a nonsquare-type
line hence it does not contain any isotropic subspaces, a contradiction.The
case when $W$ is of nonsquare-type is identical. \end{proof}
\begin{lem}
If $W$ is a $3$-dimensional subspace of $V$ with $1$-dimensional
radical, then all $2$-dimensional subspaces of $W$ not containing
the radical are nondegenerate of the same type. \end{lem}
\begin{proof}
Let $\gen{r}=\rad{W}$ and let $U$ be a $2$-dimensional subspace
of $W$ not containing $r$. If $v\in\rad{U}$ then $\gen{v,r}$ is
in the radical of $W$, a contradiction. Finally assume $U,U'$ are
$2$-dimensional subspaces of $W$ not containing the radical. If
$U=\gen{e,f}$ then there are $\alpha,\beta\in\Fq$ with $U'=\gen{e+\alpha r,f+\beta r}$
and it is quite obvious now that the two bases have the same Gram
matrix. 
\end{proof}

\section{Connectedness}

\label{con}

A geometry is connected if and only if its collinearity graph is connected.
The diameter of a geometry is the maximum of the distances between
points in the collinearity graph of the geometry. 
\begin{lem}
If $n\ge3$ then $\Gamma$ is connected of diameter $2$. \end{lem}
\begin{proof}
Let $a$ and $b$ be square-type points. We will consider three different
cases depending on what type of subspace $\gen{a,b}$ is. \\
 Case 1: $\gen{a,b}$ is of square-type. Then $a$ and $b$ are
collinear and hence connected with distance $1$. \\
 Case 2: $\gen{a,b}$ is of nonsquare-type. Then $\gen{a,b}^{\perp}$
is also of nonsquare-type and hence has at least one square-type point
$c$. This means that $\gen{a,c}$ and $\gen{a,b}$ are both square-type
lines. Then $a$ is connected to $b$ with distance $2$. \\
 Case 3: $\gen{a,b}$ is a degenerate subspace. Let $\gen{e}=\mbox{rad}(\gen{a,b})$
Since $\gen{a}^{\perp}$ is nondegenerate we can complete $e$ to
a hyperbolic basis $\{e,f\}$ in $\gen{a}^{\perp}$. Finally we notice
that $\gen{a,e,f}$ is nondegenerate of square-type so we can let
$c\in\gen{a,e,f}^{\perp}$ of square-type. It follows that $c$ is
collinear to both $a$ and $b$. \end{proof}
\begin{lem}
\label{lem:connect} Let $n=2$. If $q\geq7$ \inputencoding{latin1}{then
$\Gamma$ is connected }\inputencoding{latin9}with diameter three. \end{lem}
\begin{proof}
Let $a$ and $b$ be points. We again consider three separate cases.
\\
 Case 1: $\gen{a,b}$ is a square-type line. There is nothing to
show here. \\
 Case 2: $\gen{a,b}$ is a degenerate line. Let $\gen{e}=\rad{\gen{a,b}}$
and choose $f\in\gen{a}^{\perp}$ such that $\gen{e,f}$ is a plus-type
line. If necessary scale $f$ such that $(b,f)=1$. Notice $(b,f)\neq0$
else $\gen{a}^{\perp}=\gen{b}^{\perp}$ which would force $a$ and
$b$ to be the same point. Now scale $e$ so that $\{e,f\}$ is a
hyperbolic basis. Choose $\beta\neq0$ such that $1-\beta^{2}$ is
a square which we can do since $q\geq7$ and pick $\alpha=1/(2\beta)$.
Consider the point $x=\alpha e+\beta f$. We have that $(x,x)=2\alpha\beta=1$
and so $x$ is a square-type point. Also the Gram matrix for $\gen{a,x}$
is $\four{1}{0}{0}{1}$ with determinant one and the Gram matrix for
$\gen{b,x}$ is $\four{1}{\beta}{\beta}{1}$ with determinant $1-\beta^{2}$.
So both of these lines are of square-type. The conclusion thus follows.
\\
 Case 3: $\gen{a,b}$ is a nonsquare-type line. We note that $\gen{b}^{\perp}$
is a square-type line and hence in this case a plus-type line. Let
$s$ be an isotropic vector in $\gen{b}^{\perp}$. We consider the
space $\gen{b,s}$ which is degenerate but must intersect $\gen{a}^{\perp}$
since each is a two dimensional subspace of a three dimensional space.
Notice $s\not\in\gen{a}^{\perp}$ else $V$ would be degenerate. There
are only $q$ other points on $\gen{b,s}$ all of which are square-type.
Thus $\gen{b,s}\cap\gen{a}^{\perp}$ is a square-type point call it
$c$. We now have $\gen{b,c}$ is a degenerate two dimensional subspace
which means $b$ can be connected to $c$ with distance two and since
$c$ is in $\gen{a}^{\perp}$ it is connected to $a$ with distance
one. Thus the diameter is three. \\

\end{proof}
Before proceeding to simple connectedness we verify the desired group
acts flag transitively on the geometry $\Gamma$. Fix $H:=O_{n}^{+}(q)$.
\begin{lem}
$H$ acts flag transitively on $\Gamma$.\end{lem}
\begin{proof}
Let $V_{1}\subset V_{2}$ and $W_{1}\subset W_{2}$ be subspaces such
that $V_{i}\iso W_{i}$. Fix $\psi:V_{2}\to W_{2}$ an isometry and
let $\varphi:V_{1}\to W_{1}$ be an arbitrary isometry. Notice that
$\psi(V_{1})\subset W_{2}$ and so we may apply Witt's Lemma and obtain
an isometry $\overline{\varphi}:W_{2}\to W_{2}$ such that $\overline{\varphi}|\psi(V_{1})=\varphi\circ\psi^{-1}$.
Finally we claim that the map $\overline{\varphi}\circ\psi:V_{2}\to W_{2}$
is the desired extension of $\varphi$. If $v\in V_{1}$ then $\overline{\varphi}(\psi(v))=\varphi(\psi^{-1}(\psi(v)))=\varphi(v)$
and so $\overline{\varphi}\circ\psi|V_{1}=\varphi$. We can now move
from one step to the next in any flag and thus flag transitivity follows
by induction. 
\end{proof}

\section{Simple Connectedness (case $n\ge4$)}

\label{nge4} We assume we are in the case when $n\ge4$. The following
are rather standard results about geometries.
\begin{defn}
A cycle is geometric if it is fully contained in $\{a\} \cup \res{a}$
for some $a\in\Gamma$.  \end{defn}
\begin{lem}
Every geometric cycle is homotopic to the trivial cycle. \end{lem}
\begin{proof}
If the cycle is contained in $\{a\}\cup\res{a}$ then the cycle is
homotopic to the constant cycle by first inserting returns to $a$
and then removing triangles. \end{proof}
\begin{cor}
If two cycles are obtained by inserting or deleting a geometric cycle
then they are homotopic. 
\end{cor}
At this point we fix a point $x$ and let it be our base point i.e.,
we will consider $\pi_{1}(\Gamma)=\pi_{1}(\Gamma,x)$. Let $\Sigma$
be the subgraph in the incidence graph of $\Gamma$ that is induced
by the points and lines. For an element $a\in\Gamma$ that is neither
a point nor a line, define $\Sigma_{a}=\Sigma\cap\res{a}$. That is
$\Sigma_{a}$ consists of all the points and lines incident to $a$.
\begin{lem}
Every cycle starting at $x$ is homotopic to a cycle that is fully
contained in $\Sigma$. \end{lem}
\begin{proof}
Standard, see \cite{Bue95}. \end{proof}
\begin{lem}
\label{lem:three spaces into 4}Any $3$-dimensional space that contains
a line of the geometry can be embedded in a proper subspace 
of $V$ of square-type. \end{lem}
\begin{proof}
Let $W$ be the three dimensional subspace. If $W$ is of nondegenerate
plus-type there is nothing to prove. If $W$ is of nonsquare-type
then $W^{\perp}$ is a 2 dimensional space of nonsquare-type and so
it will contain a subspace $\gen{u}$ of nonsquare-type. The space
$\gen{W,u}$ is then of square-type.

Finally if $W$ is degenerate then consider $\gen{e}=\rad{W}$ (note
that the radical is $1$-dimensional by hypothesis.) Consider $L\leq W$
a line of the geometry and $f\in L^{\perp}\setminus e^{\perp}$. Such
an element exists because otherwise $V\subseteq e^{\perp}$. The space
$\gen{W,f}$ is then the orthogonal sum of the two square-type spaces
$L$ and $\gen{e,f}$ and so it is of square-type. \end{proof}
\begin{cor}
For $n\geq4$ every triangle is geometric. \end{cor}
\begin{proof}
Let $\gamma=abca$ and $W=\gen{a,b,c}$. Then $W$ is a three dimensional
space as in \ref{lem:three spaces into 4} and so the triangle is
geometric.\end{proof}
\begin{lem}
\label{lem:connected subspaces}Let $W$ be a 3-dimensional subspace
of $V$ containing at least one line of the geometry and $q\geq7$.
Then for all square-type points $a,b$ in $W$ there is a path joining
then entirely in $W$. \end{lem}
\begin{proof}
First I claim that if $W$ is a nondegenerate space then it is connected.
If $W$ is square-type then the we have already done this in Lemma
\ref{lem:connect}. If $W$ is nonsquare-type then the proof is similar
to that of Lemma \ref{lem:connect} but a little easier. Note first
that by Lemma \ref{lem:no tricky subspaces} the line $\gen{a,b}$
cannot be degenerate so we have the following 2 cases. Case 1: $\gen{a,b}$
is a square-type line. This means $a$ and $b$ can be connected.
\\
 Case 2: $\gen{a,b}$ is a nonsquare-type line. This means that
there is a square-type point in $\gen{a,b}^{\perp}$. Call this point
$x$ and so $\gen{a,x}$ and $\gen{b,x}$ are square-type lines and
so $a$ and $b$ are connected. \\

Now suppose $\dim{(\rad{W})}=1$. Then since $W$ contains a line
of the geometry, any two dimensional subspace of $W$ not containing
the radical is a line of the geometry. If $a$ and $b$ are not collinear
it means that the radical of $W$ is on $\gen{a,b}$. It then means
that $a$ and $b$ are collinear to any point not on $\gen{a,b}$. \end{proof}
\begin{lem}
For $n\geq4$ every $4$-cycle is decomposable. \end{lem}
\begin{proof}
Let $\gamma=abcda$. \\
 Claim: Without loss of generality we may assume that $b,d\in\gen{a,c}^{\perp}$.
\\
 Proof of Claim: Consider the spaces $W=\gen{a,b,c}^{\perp}$ and
$W'=\gen{a,d,c}^{\perp}$. They are codimension one subspaces of the
square-type nondegenerate subspaces $\gen{a,b}^{\perp}$ and respectively
$\gen{a,d}^{\perp}$. The later have dimension at least 3 so by Lemma
\ref{lem:no tricky subspaces} there exists $b_{1}\in W$ and $d_{1}\in W'$,
points of the geometry. Decompose $\gamma$ as a sum of the 4 triangles
$abb_{1}a,bb_{1}cb,cd_{1}dc,dd_{1}ad$ and the $4$-gon $ab_{1}cd_{1}a$.
Thus the claim is true and It follows that $\gen{a,c}\perp\gen{b,d}$.

Assume first that $\gen{a,c}$ is nondegenerate. If it is a line of
the geometry then $\gamma$ decomposes into 2 triangles so we might
assume that it is a nonsquare-type. It follows that $U=\gen{a,c}^{\perp}$
is a nondegenerate subspace of $V$ of dimension at least 3 and $b,d\in U$
and so by Lemma \ref{lem:connected subspaces} $b$ is connected to
$d$ inside of $U$. All elements of $U$ are connected to both $a$
and $c$ decomposing the $4$-gon into triangles.

Next assume that $\gen{a,c}$ is degenerate and consider again $U=\gen{a,c}^{\perp}$,
which is now a degenerate subspace of dimension at least 3 and $1$-dimensional
radical containing $b$ and $d$. If $b$ and $d$ can be embedded
in a $3$-dimensional subspace of $U$ that contains a line of the
geometry we can once again use Lemma \ref{lem:connected subspaces}
to decompose the $4$-gon.

The only bad case would be when $U$ itself is $3$ dimensional (that
is $n=3$) and contains no lines of the geometry. This is an impossibility.
Indeed let $L$ be a nondegenerate (hence nonsquare-type) subspace
of $W$. Note that $\gen{a,c}\le L^{\perp}$ and that $L^{\perp}$
is a 3-dimensional nonsquare-type space and so, by Lemma \ref{lem:no tricky subspaces}
cannot contain degenerate subspaces with square-type points like $\gen{a,c}$.\end{proof}
\begin{lem}
For $n\geq4$ and $q\ge7$ then every $5$-cycle is decomposable. \end{lem}
\begin{proof}
Let $\gamma=abcdea$ be a $5$-cycle. Consider $W=\gen{a,c,e}$ which
is a three dimensional space with radical of dimension at most one
since $\gen{a,e}$ is line. If $W$ is nondegenerate then $W^{\perp}$
is a nondegenerate subspace of dimension at least two and so it contains
at least one point $x$ of the geometry. It follows that $x$ is collinear
to $a,c,e$. This decomposes the pentagon as a sum of the triangle
$axe$ and the two $4$-gons $axcb$ and $excd$.

Similarly if $W$ is degenerate and $v$ is its radical note that
any $2$-space in $W$ not passing through $v$ is a line of the geometry.
If $c$ and $e$ are collinear then the pentagon immediately decomposes
into a triangle and a $4$-gon so we might assume $v$ is on the line
$\gen{c,e}$.  It then follows that $a$ and $c$ must be collinear
and so the pentagon decomposes. \end{proof}
\begin{cor}
If $n\ge4$ and $q\ge7$ then the geometry is simply connected. 
\end{cor}

\subsection{The case $n=3$}

\label{n=00003D3}

Here we consider the case $n=3$. Throughout this section we will
assume that $q\ge47$.  First we will decompose triangles. There are
two types of triangles, those that generate a degenerate 3-space and
those that generate a nondegenerate space of nonsquare-type. We will
call the former \emph{degenerate triangles} and the later \emph{triangles
of nonsquare-type}. We first show that we can assume that two vertexes
of any triangle are orthogonal.
\begin{lem}
Any triangle can be decomposed into geometric triangles and a triangle
in which two vertexes are orthogonal. \end{lem}
\begin{proof}
Let $v_{1}v_{2}v_{3}$ be a triangle and let $W=\langle v_{1},v_{2},v_{3}\rangle$.

First assume that we deal with a degenerate triangle. Note that in
the degenerate space $W$, any line that does not pass through the
radical is a line of the geometry. Indeed the determinant of the Gram
matrix for any such line will be the same as the one for any edge
of the triangle. Moreover let $v$ be the unique point on $\langle v_{2},v_{3}\rangle$
that is perpendicular to $v_{1}$. The point $v$ cannot be isotropic
because then it would have to equal the radical. The line $\langle v_{1},v\rangle$
cannot pass through the radical (because it would then be degenerate
and that's not the case) and so it would have to be a line of the
geometry. It follows that $v$ is a point of the geometry and you
can split the initial triangle into $v_{1}vv_{2}$ and $v_{1}vv_{3}$.

Now assume that the triangle is of nonsquare-type.  Take $e\in W\cap\langle v_{2},v_{3}\rangle^{\perp}$.
This is a nonsquare point. As above let $v$ be the point on $\langle v_{2},v_{3}\rangle$
that is perpendicular to $v_{1}$. If $v$ is isotropic then the line
$\langle v_{1},v \rangle$ equals $v^{\perp}$ and so it must contain $e$. However
it is a degenerate line so it only contains points of one type, a
contradiction. Therefore $v$ must be nonisotropic. If it is of square-type 
we are done.

We can therefore assume that $v$ is a nonsquare point. Let $f$ a
point such that $e,f$ is an orthogonal basis for $\langle v_{2},v_{3}\rangle^{\perp}$.
Since this line is of square-type we can assume that $(e,e)=(f,f)=t$
a nonsquare element. Note that $\langle f \rangle =W^{\perp}$. We claim that $v_{1}$
is collinear to at least one point on $\langle e,f\rangle$. To see
that we will look inside the space $U=\langle v_{1},e,f\rangle$.
Notice that $f\perp\langle v_{1},e\rangle$ and $\langle v_{1},e\rangle=v^{\perp}\cap W$
and so it is a square-type line.  Therefore the space $U$ is a nonsquare
space.

Let $v_{4}=xe+yf$ then $(v_{4},v_{4})=(x^{2}+y^{2})t$ and the Gram
matrix for the line $v_{1},v_{4}$ is \[
\left(\begin{array}{cc}
(x^{2}+y^{2})t & (e,v_{1})\\
(e,v_{1}) & 1\end{array}\right)\]
 and so its determinant is $(x^{2}+y^{2})t^{2}-(e,v_{1})^{2}$. By
Corollary \ref{cor:sum of squares} there exists $a\in\Fq$ so that
$a^{2}-(e,v_{1})$ is a square and for that $a$, there exists $x,y\in\Fq$
so that $(x^{2}+y^{2})t=a^{2}$. This decomposes the initial triangle
as a sum of one degenerate triangle $v_{2}v_{3}v_{4}$ and two triangles
$v_{4}v_{1}v_{2}$ and $v_{4}v_{1}v_{3}$in which two vertexes are
orthogonal. 
\end{proof}
The following lemma decomposes all degenerate triangles. 
\begin{lem}
\label{lem:degenerate triangles} Any degenerate triangle $e_{1},e_{2},v$
such that $(e_{1},e_{2})=0$ is homotopically trivial. \end{lem}
\begin{proof}
We can assume that $e_{1}$ and $e_{2}$ have norm one. Let $e,f$
be a hyperbolic basis for $\langle e_{1},e_{2}\rangle^{\perp}$ such
that $e\in W$ and that $v=\alpha e_{1}+\beta e_{2}+e$ with $\alpha^{2}+\beta^{2}=1$.
We construct the points $v_{1}=e_{1}-\alpha f$, $v_{2}=e_{2}-\beta f$.

Note that $v_{2}$ is perpendicular (hence collinear) to $e_{1}$
and $v$ and that $v_{1}$ is perpendicular to $e_{2}$ and $v$.
We will construct $v_{3}=xe+yf$ such that it will complete the icosahedron.
The conditions are as follows.

The lines $\gen{v_{3},v_{1}}$ and $\gen{v_{3},v_{2}}$ have Gram
matrix $\left(\begin{array}{cc}
2xy & -\alpha x\\
-\alpha x & 1\end{array}\right)$ and respectively $\left(\begin{array}{cc}
2xy & -\beta x\\
-\beta x & 1\end{array}\right)$so the determinants are $2xy-\alpha^{2}x^{2}$ respectively $2xy-\beta^{2}x^{2}$.

The lines $v_{3},e_{i}$ have diagonal Gram matrices and determinant
$2xy$.

The triangles $(e_{1},v_{2},v)$, $(v,v_{1},v_{2})$ and $(e_{2},v_{1},v)$
are geometric.

The plane of the triangle $(v_{3},v_{1},v_{2})$ has Gram matrix \[
\left(\begin{array}{ccc}
2xy & -\alpha x & -\beta x\\
-\alpha x & 1 & 0\\
-\beta x & 0 & 1\end{array}\right)\]
 and so it has determinant $2xy-x^{2}$.

The plane of the triangle $(v_{3},e_{1},v_{2})$ has Gram matrix \[
\left(\begin{array}{ccc}
2xy & 0 & -\beta x\\
0 & 1 & 0\\
-\beta x & 0 & 1\end{array}\right)\]

and determinant $\frac{1}{\alpha^{2}}2xy-x^{2}$.

The plane of the triangle $(v_{3},e_{2},v_{1})$ has Gram matrix \[
\left(\begin{array}{ccc}
2xy & 0 & -x\\
0 & 1 & 0\\
-x & 0 & 1\end{array}\right)\]
 of determinant $2xy-\beta^{2}x^{2}$.

Therefore the conditions we need to satisfy are: $2xy,2xy-x^{2},2xy-\alpha^{2}x^{2},2xy-\beta^{2}x^{2}$
are all squares.

We can pick $y=-\frac{1}{2x}$ to get that the conditions are (since
-1 is a square) $1+x^{2},1+\alpha^{2}x^{2},1+\beta^{2}x^{2}$ are
all squares. In particular if we denote $c=x,a=\alpha x,b=\beta x$
we note that the conditions are exactly the ones in Lemma \ref{lem:Joe's lemma}
and in particular each value of $x$ will give a decomposition of
all possible triangles that involve a certain choice of
$\alpha$ and $\beta$ depending on $x$. Note also that there are
$\frac{q-1}{2}$ choices for $x$ and just as many choices for the
point $(\alpha,\beta)$. This completes the proof. \end{proof}
\begin{lem}
Any nonsquare-type triangle $e_{1}e_{2}v$ with $e_{1}\perp e_{2}$
is null homotopic. \end{lem}
\begin{proof}
Consider $e$ an isotropic vector in $\gen{e_{1},e_{2}}^{\perp}$
and $W=\gen{e_{1},e_{2},e}$. Note that $e\not\perp v$ and so $W\cap v^{\perp}$
is a line of the geometry. Consider $v_{1},v_{2}\in W\cap v^{\perp}$
so that $v_{i}\perp e_{i}$ and $v_{1}\perp v_{2}$. Note that $v_{i}$
are points of the geometry and that they are collinear to $v$. Moreover
the triangles $vv_{i}e_{i}$ and $vv_{1}v_{2}$ are all geometric.
It follows that our triangle $e_{1},e_{2},v$ is homotopic to the
4-gon $e_{1}e_{2}v_{2}v_{1}$. Pick a point $x$ of the geometry on
the line $\gen{v_{1},v_{2}}$. Since $\gen{e_{1},v_{2}}$ and $\gen{e_{2},v_{1}}$
both contain $e$, the radical of $W$, it follows that $\gen{x,e_{i}}$
cannot contain it so they must be lines of the geometry. This means
that the 4-gon $e_{1}e_{2}v_{2},v_{1}$ can be decomposed into triangles
$e_{1}e_{2}x$, $e_{2}xv_{2}$ and $xv_{1}e_{1}$, all of which are
degenerate with a pair of perpendicular vertexes hence null homotopic
by Lemma \ref{lem:degenerate triangles} \end{proof}
\begin{lem}
\label{lem:degenerate 4-gons} Any 4-gon that lies in a degenerate
$3$-space is null homotopic. \end{lem}
\begin{proof}
Let $e_{1}e_{2}e_{3}e_{4}$ be such a 4-gon, $W=\gen{e_{1},e_{2}.e_{3},e_{4}}$
and $v$ be the radical of $W$. Any $2$-dimensional space of $W$
not passing through $v$ is a line of the geometry. The line $\gen{e_{1},e_{2}}$
contains $\frac{q-1}{2}$ points of the geometry and all but possibly
2 are collinear with both $e_{3}$ and $e_{4}$. So if $q\ge7$ there
exists $e$ on this line so that the $4$-gon decomposes as 3 triangles
$e_{1}e_{2}e$, $e_{2}e_{3}e$ and $e_{1}e_{4}e$. \end{proof}
\begin{lem}
Any $4$-gon is null homotopic. \end{lem}
\begin{proof}
Let $e_{1}e_{2}e_{3}e_{4}$ be a 4-gon that is not contained in a
degenerate $3$-space. Consider $e$ and $f$ two $1$-dimensional
isotropic spaces so that $e\perp\gen{e_{1},e_{2}}$ and $f\perp\gen{e_{3},e_{4}}$.
Consider $W=\gen{e_{1},e_{2},e}$ and $W'=\gen{e_{3},e_{4},f}$ and
let us look at $W\cap W'$. Note first that since $\gen{e_{1},e_{2}.e_{3},e_{4}}$
is not a degenerate $3$-space, $w\ne W'$. If $W\cap W'$ contains
$e$ or $f$ then all the $1$-spaces on $W\cap W'$ are points of
the geometry. Moreover all but possibly two are collinear with all
of $e_{i}$ decomposing the $4$-gon into triangles. So we might assume
that neither $e$ nor $f$ lie on $W\cap W'$ and so the latter is
a line of the geometry. As before if $q\ge7$, there exists two points
of the geometry $v,v'\in W\cap W'$ so that $v$ is collinear to $e_{1}$
and $e_{4}$ and $v'$ is collinear to $e_{2}$ and $e_{3}$. This
decomposes the $4$-gon as a product of 2 triangles $( e_{1},e_{4},v)$
and $(e_{2},e_{3},v')$ and two $4$-gons $(e_{1},e_{2},v',v)$
and $(e_{3},e_{4},v,v')$. Note moreover that both $4$-gons are as
in Lemma \ref{lem:degenerate 4-gons} and so are null homotopic finishing
the proof. \end{proof}
\begin{lem}
Every pentagon is null homotopic. \end{lem}
\begin{proof}
Let $p_{1},..p_{5}$ be a pentagon. Consider the lines $l=<p_{2},p_{3}>^{\perp}$
and $l'=<p_{4},p_{5}>^{\perp}$. They are nondegenerate lines of plus-type 
and so they each contain two isotropic points.

Assume that we can find $v\in l$ and $v'\in l'$ isotropic vectors
such that $v\not\perp v'$ then the spaces $W=<v,p_{2},p_{3}>$ and
$W'=<v',p_{4},p_{5}>$ are both degenerate of rank two and $v$ respectively
$v'$ are their radicals. It follows as above that any line not passing
through the radicals is a line of the geometry. In particular the
intersection $W\cap W'$ is a line of the geometry. Moreover each
of $p_{2},p_{3},p_{4},p_{5}$ are collinear to all but one point of
$W\cap W'$. In particular if $\frac{q-1}{2}>3$ that is $q>7$ there
exist points $v_{1},v_{2}$ on $W\cap W'$ such that $v_{1}$ is collinear
to $p_{2}$ and $p_{5}$ and $v_{2}$ is collinear to $p_{3}$ and
$p_{4}$.  The pentagon can then be decomposed into three $4$-gons
$(p_{1},p_{2},v_{1},p_{5})$, $(p_{2},v_{1},v_{2},p_{3})$, $(p_{4},v_{2},v_{1},p_{5})$
and a triangle $(p_{3},v_{2},p_{4})$.

The remaining case is when any pair of isotropic points from $l$
and $l'$ are orthogonal. This means $l\perp l'$ and so $l'=l^{\perp}=<p_{2},p_{3}>$
respectively $l=l'^{\perp}=<p_{4},p_{5}>$. In particular $p_{2}\perp p_{5}$
and so the pentagon decomposes into a $4$-gon $(p_{2},p_{3},p_{4},p_{5})$
and a triangle $(p_{1},p_{2},p_{5})$. \end{proof}
\begin{cor}
If $n=3$ and $q\ge47$ then the geometry is simply connected. 
\end{cor}

\section{The residual geometries}

We still need to discuss the residual geometries. Let $R$ be a residue
and $J$ its type set. If $J=J_{1}\cup J_{2}$ where $\max\left\{ i\in J_{1}\right\} +1\le\min\left\{ i\in J_{2}\right\} $
then the residual geometry is the product of the two residual geometries.
The product of two geometries is connected if both are nonempty and
is simply connected if one of the geometries is nonempty and the other
is connected. Therefore we need to investigate the case $J=\{i+1,\cdots i+k\}$.
If $U_{i}$ and respectively $U_{i+k+1}$ are the object of type $i$
respectively $i+k+1$ in the flag that defines the residue then the
residual geometry is isomorphic to the geometry $\Gamma$ on the space
$U_{i}^{\perp}\cap U_{i+k+1}$.  The second part of Theorem \ref{thm2}
follows from the results we have just proved.

\end{document}